\documentclass[a4paper,12pt]{amsart}
\usepackage{amsmath}
\usepackage{amsthm}
\usepackage{amssymb}
\usepackage{fullpage}
\usepackage{bbm}
\usepackage{verbatim}

\newtheorem{thm}{Theorem}

\newtheorem{prop}{Proposition}
\newtheorem{lem}{Lemma}

\newcommand{\Setmake}[2]{\{#1 : #2 \}}

\newcommand{\tRe}{\textup{Re }}

\newcommand{\bfrac}[2]{\left(\frac{#1}{#2}\right)}
\newcommand{\dv}{\boldsymbol d}

\newcommand{\zv}{\boldsymbol z}

\newcommand{\W}{\mathcal W}

\newcommand{\B}{\mathcal B}
\newcommand{\cU}{\mathcal U}
\newcommand{\A}{\mathcal A}
\newcommand{\V}{\mathcal V}

\newcommand{\ep}{\varepsilon}
\newcommand{\beql}[1]{\begin{equation}\label{#1}}
\newcommand{\eeq}{\end{equation}}

\begin{document}
\title{Almost prime triples and Chen's theorem} 
\author{Roger Heath-Brown}
\address{Mathematical Institute \\
University of Oxford\\
Andrew Wiles Building \\
Radcliffe Observatory Quarter\\
Woodstock Road\\
Oxford\\UK\\
OX2 6GG }
\email{Roger.Heath-Brown@maths.ox.ac.uk}
\author{Xiannan Li}
\address{Mathematical Institute \\
University of Oxford\\
Andrew Wiles Building \\
Radcliffe Observatory Quarter\\
Woodstock Road\\
Oxford\\UK\\
OX2 6GG }
\email{lix1@maths.ox.ac.uk}


\begin{abstract} 
We show that there are infinitely many primes $p$ such that not only
does $p+2$ have at most two prime factors, but $p+6$ also has a
bounded number of prime divisors.  This refines the well known result
of Chen \cite{Ch1}.
\end{abstract}

\maketitle

\section{Introduction} 

The twin prime conjecture states that there are infinitely many primes
$p$ such that $p+2$ is also prime.  Although the conjecture has
resisted our efforts, there has been spectacular partial progress.
One well known result is Chen's theorem \cite{Ch1} that there are
infinitely many primes such that $p+2$ has at most two prime factors.
In a different direction, building on the work of Goldston, Pintz, and
Y\i ld\i r\i m \cite{GPY}, it has recently been shown by Zhang \cite{Zh}
that there are
bounded gaps between consecutive primes infinitely often.  The
numerical result has been improved in the works of the Polymath8
project \cite{Poly} and Maynard \cite{JM}, and the bounded gaps result
has also been extended to prime tuples by Maynard \cite{JM} and Tao
(unpublished). 

The twin prime conjecture is a special case of the Hardy-Littlewood
conjecture, which postulates asymptotics for prime tuples in general.
An example is that one expects that the number of primes $p\leq x$ such
that $p+2$ and $p+6$ are simultaneously prime should be asymptotic to  
$$C \frac{x}{\log^3 x}$$ 
for a certain positive constant $C$ (given by
\eqref{eqn:C}).  In this direction, it has been proven that there are
infinitely many natural numbers $n$ such that $n(n+2)(n+6)$ is almost
prime --- that is, $n(n+2)(n+6)$ has at most $r$ prime factors, for some
finite $r$.  More specifically, Porter \cite{Port} proved this
statement for $r = 8$ and this was improved by Maynard \cite{JM2} to
$r=7$.   

We are interested in proving an analogue of Chen's theorem for prime
tuples.  More precisely, we show that there are infinitely many
primes $p$ such that $p+2$ has at most two prime factors, and $p+6$
has at most $r$ prime factors for some finite $r$.

\begin{thm}
Let $\pi_{1, 2, r}(x)$ denote the number of primes $p\leq x$ such that
$p+2$ has at most two prime factors and $p+6$ has at most $r$ prime
factors.  Then 
\begin{equation}
\pi_{1, 2, r} (x) \gg \frac{x}{\log^3x}
\end{equation}for $r=76$.
\end{thm}

Our basic philosophy, which the proof will illustrate, is the
following.  Suppose one has polynomials $f_1(x),\ldots,f_{k+1}(x)$ and
positive integers $r_1,\ldots,r_k$.
Then, if the weighted sieve can prove that
\[f_1(n)=P_{r_1},\ldots,f_k(n)=P_{r_k}\]
for infinitely many integers
$n$, then one should be able to modify the argument to show the
existence of a positive integer $r_{k+1}$ such that
\[f_1(n)=P_{r_1},\ldots,f_{k+1}(n)=P_{r_{k+1}}\]
for infinitely many integers $n$.

Our approach uses the weighted sieve which appeared in Chen's original
work, as well as the vector sieve of Br\"{u}dern and Fouvry \cite{BF}.  We
will also use a Selberg upper bound sieve  of ``mixed dimension''. The
value of $r$ in our theorem could be improved by using a more 
elaborate weighted sieve, but we will not pursue this.

{\bf Acknowledgement.} We would like the thank the referee for a very careful reading of the manuscript, with special attention to the numerical work described at the end of Section 4.  In particular, we are indebted to the referee for the neat calculation of the values of $B(1, (4\theta_2)^{-1})$ and $J(\theta_1, \theta_2, \theta)$. This work was supported by EPSRC grant 
EP/K021132X/1. 

\section{The basic setup}

In the sequel, $p$ and $p_i$ shall always denote primes.  Let $\ep>0$
be a small positive constant, and let $x^\ep<\xi_2<\xi_1\le x^{1/3}$ be
parameters to be decided in due course. We will work with the set 
\[\A=\{p+2:\, x^{1/3}<p\le x-6,\, (p+6,P(\xi_2))=1\},\]
where 
\[P(w)=\prod_{p<w}p\]
as usual.

The basic idea in Chen's argument is to consider the expression
\beql{eqn:basicdecomp}
S_1=S(\A;\xi_1)-\frac12 \sum_{\xi_1\le p\leq x^{1/3}}S(\A_p;\xi_1)-\frac12
N_0,
\eeq
where
\[N_0=\#\{p_1p_2p_3\in\A:\, \xi_1\le p_1\le x^{1/3}<p_2<p_3\}.\]
One then has an inequality of the form
\beql{chin}
S_1\le \#\A^{(0)}+\#\A^{(1)}+\frac12\#\A^{(2)},
\eeq
where
\[\A^{(0)}=\{n\in\A:\, p_1^2\mid n \mbox{ for some }p_1\ge \xi_1\},\]
\[\A^{(1)}=\{n\in\A:\, n \mbox{ prime, or } 
n=p_1p_2\mbox{ with }p_2>p_1>x^{1/3}\},\]
and
\[\A^{(2)}=\{n\in\A:\, n=p_1p_2\mbox{ with }p_2>x^{1/3}\ge p_1\ge \xi_1\}.\]
The bound (\ref{chin}), which the reader may easily verify, is closely related
to the inequality used by Halberstam and Richert \cite[Chapter 11,
(2.1)]{HR}, for example.

One immediately has $\#\A^{(0)}\ll x/\xi_1$, which will be sufficiently
small for our purposes.  Moreover one can see that if
$n\in\A^{(1)}\cup\A^{(2)}$ then $n=p+2$ with $n=P_2$ and $n+4=P_r$,
where $r=[(\log x)/(\log \xi_2)]$.  We therefore obtain a result of the
type given in our theorem provided that we can give a suitable
positive lower bound for $S_1$. This can
be achieved by using the vector sieve of Br\"{u}dern and Fouvry \cite{BF}
in place of the usual upper and lower bound sieves.

There are a number of methods to try to improve the value of $r$
obtained by this naive approach.  We
choose to include a simple weighted sieve in order to 
eliminate those triples $(p, p+2, p+6)$ for which $p+6$ has many prime
factors. (The reader will observe that one could do better by 
incorporating more elaborate weights into (\ref{eqn:basicdecomp}).)

We proceed to define the sets
\[\B^{(i)}=\{p+6:\, p+2\in\A^{(i)}\},\;\;\;(i=1,2)\]
and a weight function
\beql{wtd}
w_p=1-\frac{\log p}{\log y}
\eeq
where $y=x^{1/v}$ for some positive constant $v$ to be decided in due
course. At this stage we insist only that $\xi_2<y<x$. Since any element
of $\B^{(i)}$ is coprime to $P(\xi_2)$, and since $w_p<0$ for $p>y$, we now have
\begin{eqnarray*}
\sum_{\xi_2\le p\le y}w_p\#\B_p^{(i)}&\ge &\sum_{2\le p\le x}w_p\#\B_p^{(i)}\\
&=&\sum_{b\in\B^{(i)}}
\left(\omega(b)-\frac{1}{\log y}\sum_{p\mid b}\log p\right)\\
&\ge&\sum_{b\in\B^{(i)}}\left(\omega(b)-\frac{1}{\log y}\log
  x\right)\\
&=&\sum_{b\in\B^{(i)}}(\omega(b)-v).
\end{eqnarray*}
Here, as usual, $\omega(b)$ denotes the number of distinct prime factors of $b$.  It then follows from (\ref{chin}) that if $\lambda$ is any positive
constant then
\begin{eqnarray}\label{eqn:weightedsecondcompbdd}
\lefteqn{S_1-
\lambda\sum_{\xi_2\le p\le y}w_p\left(\#\B_p^{(1)}+\frac12\#\B_p^{(2)}\right)}
\hspace{1cm} \notag \\
&\le&
O(x/\xi_1)+\left(\sum_{b\in\B^{(1)}}+\frac12\sum_{b\in\B^{(2)}}\right)
(1+\lambda v-\lambda\omega(b))\notag \\
&\le&
O(x/\xi_1)+(1+\lambda v)\#\{b\in\B^{(1)}\cup\B^{(2)}:\,\omega(b)<\lambda^{-1}+v\} \notag  \\
&\le&O(x/\xi_2)+(1+\lambda v)
\#\{b\in\B^{(1)}\cup\B^{(2)}:\,\omega(b)<\lambda^{-1}+v,\,b\mbox{ square-free}\}.
\end{eqnarray}
Here we use the observation the the number of elements of $\B^{(i)}$
which are not square-free must be $O(x/\xi_2)$, since any such element
is coprime to $P(\xi_2)$ by definition.

We therefore seek to show that
\beql{goal}
S_1-\lambda\sum_{\xi_2\le p\le y}w_p
\left(\#\B_p^{(1)}+\frac12\#\B_p^{(2)}\right)
\ge \{c+o(1)\}\frac{x}{(\log x)^3}
\eeq
for some positive constant $c$. Substituting our expression for $S_1$ from \eqref{eqn:basicdecomp}, we see that we must bound $S(\A;\xi_1)$ from below, which we accomplish using a combination of the linear
sieve with the vector sieve.  We require upper bounds for the rest of
the terms.  Here, we use two distinct methods. For
\[\sum_{\xi_1\le p\leq x^{1/3}}S(\A_p;\xi_1)\]
we will use the vector sieve for some ranges of $p$ and the Selberg sieve for other ranges of $p$.  For the remaining terms
it turns out to be more efficient to apply the Selberg sieve.  Our application has the novel feature that the sieving
dimension changes from 2 (for primes $p<\xi_2$) to 1 (for larger primes)
part way through the range.  Naturally, for the term $N_0$ we first 
apply Chen's famous ``reversal of r\^oles'' trick before applying the
upper bound sieve.

\section{Sieving tools}\label{sieve}

\subsection{The linear sieve}\label{LS}

In the Rosser--Iwaniec linear sieve one has a real parameter
$D\ge 2$ and constructs coefficients $\lambda^{\pm}(d)$ supported on
the positive integers $d\le D$ such that
\[\lambda^{\pm}(d)=\mu(d)\mbox{ or $0$, for all }d\le D\]
and
\[\sum_{d\mid n}\lambda^-(d)\le\sum_{d\mid n}\mu(d)\le
\sum_{d\mid n}\lambda^+(d)\]
for all positive integers $n|P(z)$ for some parameter $z$.  Suppose we have a multiplicative 
function $h(d)\in[0,1)$ such that
\beql{d1}
\prod_{w\le p<z}\big(1-h(p)\big)^{-1}\le 
\frac{\log z}{\log w}\left(1+\frac{L}{\log w}\right)
\eeq
for $z\ge w\ge 2$, for some parameter $L$.  Then, by Theorem 11.12 of
Friedlander and Iwaniec \cite{FI}, we have
\beql{lsub1}
\sum_{d\mid P(z)}\lambda^+(d)h(d)\le 
\left\{F(s)+O_L\left((\log
    D)^{-1/6}\right)\right\}V(z,h)\;\;\;\;\;(s\ge 1)
\eeq
and
\[\sum_{d\mid P(z)}\lambda^-(d)h(d)\ge 
\left\{f(s)+O_L\left((\log
    D)^{-1/6}\right)\right\}V(z,h)\;\;\;\;\;(s\ge 2)\]
where $F(s)$ and $f(s)$ are the standard upper and lower bound
functions for the linear sieve, with $s=(\log D)/(\log z)$, and
\[V(z,h)=\prod_{p<z}\left(1-h(p)\right).\]

Moreover one sees from \cite[(6.31)--(6.34)]{FI} that
\beql{lsub2}
\sum_{d\mid P(z)}\lambda^-(d)h(d)\le V(z,h)\le
\sum_{d\mid P(z)}\lambda^+(d)h(d).
\eeq

\subsection{The Fundamental Lemma sieve}

Let $\cU$ be a set of positive integers, possibly with multiplicities,
and suppose that
\beql{Ua}
\#\cU_d=h^*(d)Y+r(d)
\eeq
for some multiplicative function $h^*(d)\in[0,1)$.  
We assume for simplicity that
\beql{r*b}
h^*(p)\le C_0p^{-1},
\eeq
for some constant $C_0\ge 2$.  Then 
\[\prod_{w\le p<z}\big(1-h^*(p)\big)^{-1}\le 
K\left(\frac{\log z}{\log w}\right)^{\kappa}\]
for $z\ge w\ge 2$ for appropriate constants $K$ and $\kappa$ depending
only on $C_0$. Hence Corollary 6.10 of Friedlander and Iwaniec
\cite{FI} applies, and yields
\[S(\cU;z)=\{1+O_{C_0}(e^{-s})\}YV(z,h^*)+O(\sum_{d<z^s}|r(d)|),\]
for $z\ge 2$ and $s\ge 1$.

An inspection of the proof makes it clear that one only uses
(\ref{Ua}) for values $d\mid P(z)$.

\subsection{The vector sieve}
Let $\W$ be a finite subset of $\mathbb{N}^2$.
Suppose that $z_1, z_2 \geq 2$ with
\[\log z_1\asymp\log z_2\]
and write $\zv = (z_1,
z_2)$.  For $\dv = (d_1, d_2)$ and $\mathbf{n} = (n_1, n_2)$, we write 
$\dv|\mathbf{n}$ to mean that $d_i|n_i$ for $1\le i \le 2$.  Define as usual 
$$\W_{\dv} = \{\mathbf{n}\in \W: \dv|\mathbf{n}\},$$
and
$$S(\W; \zv) = \Setmake{(m, n)\in \W}{(P(z_1), m) = (P(z_2), n) = 1}.$$
Suppose that
$$\#\W_{\dv} = h(\dv)X + r(\dv)
$$
for some multiplicative function $h(\dv)\in(0,1]$ such that
$h(p,1)+h(1,p) - 1 <h(p,p)\le h(p,1)+h(1,p)$ for all primes $p$ and
\beql{rb}
h(p,1),h(1,p)\le C_1p^{-1},\quad\mbox{and}\quad 
h(p,p)\le C_1p^{-2}
\eeq
for some constant $C_1\ge 2$.  Then 
using the vector sieve and the linear sieve, we will derive both upper
and lower bounds for $S(\W; \zv)$.  

For $i=1, 2$, let $\lambda_i^+$
and $\lambda_i^-$ denote the coefficients of the upper and lower bound
linear sieves of level 
\[D_i=z_i^{s_i},\;\;\;(i=1,2)\] 
where $1\le s_i\ll 1$.   
Further, let $\delta = \mu * 1$, $\delta_i^+ = \lambda_i^+*1$ and
$\delta_i^- = \lambda_i^- * 1$.   Note that $\delta_i^- \leq \delta
\leq \delta_i^+$, and that
\beql{vu}
\delta(m)\delta(n) \leq  \delta_1^+(m) \delta_2^+(n) 
\eeq
and
\beql{vl}
\delta(m) \delta(n) \geq \delta_1^-(m)\delta_2^+(n) + \delta_1^+(m)
\delta_2^-(n) - \delta_1^+(m)\delta_2^+(n),
\eeq
for any natural numbers $m$ and $n$.

In applying the vector sieve we will want to replace $h(\dv)$ by
$h_1(d_1)h_2(d_2)$, where
\[h_1(d)=h(d,1),\quad\mbox{and}\quad h_2(d)=h(1,d).\]
There is no difficulty when $d_1$ and $d_2$ are coprime, but there are
potential problems when they share a common factor.
We circumvent this issue by using a
preliminary application of the Fundamental Lemma sieve.  Suppose we
are given $z_0\ge 2$ and  positive 
integers $d_1,d_2$ coprime to $P(z_0)$.  Let $\cU=\cU(\dv)$ be
the set of products $mn$ as $(m,n)$ runs over $\W_{\dv}$, the values $mn$
being counted according to multiplicity.  Then if $d\mid P(z_0)$ we see using the multiplicativity of $h$ that (\ref{Ua}) holds with $Y=h(\dv)X$, 
\[h^*(d)=\sum_{d=e_1 e_2 e_3}h(e_1e_3,e_2e_3)\mu(e_3)\]
and
\[r(d)=\sum_{d=e_1 e_2 e_3}r(d_1e_1e_3,d_2e_2e_3)\mu(e_3).\]
In particular 
$h^*(p)=h(p,1)+h(1,p)-h(p,p) \in [0, 1)$, and
(\ref{r*b}) holds with suitable $C_0 = 2C_1$.
The Fundamental Lemma sieve therefore shows that
\begin{eqnarray}\label{FLa}
S\big(\W_{\dv};(z_0,z_0)\big)&=&S(\cU(\dv),z_0)\nonumber\\
&=&h(\dv)XV(z_0,h^*)+O(h(\dv)Xe^{-s})
+O\left(\sum_{e_1e_2e_3<z_0^s}|r(d_1e_1e_3,d_2e_2e_3)|\right).
\end{eqnarray}

We can now apply the upper bound vector sieve.  Suppose that
$z_1,z_2\ge z_0$, and define
\[P(z_0,z)=\prod_{z_0\le p<z}p.\]
Let 
\[\W^*=\{(m,n)\in\W:\,\big(mn,P(z_0)\big)=1\}.\]
Then according to (\ref{vu}) we have
\begin{eqnarray*}
S(\W; \zv)
&=& \sum_{(m,n)\in \W^*}\delta\big((m,P(z_0,z_1)\big)
\delta\big((n,P(z_0,z_2)\big) \\ 
&\le&\sum_{(m,n)\in\W^*}\left(\sum_{d_1|(m,P(z_0,z_1))}\lambda_1^+(d_1)\right)
\left(\sum_{d_2|(n,P(z_0,z_2))}\lambda_2^+(d_2)\right)\\ 
&= &\sum_{d_1|P(z_0,z_1)}\;\; \sum_{d_2|P(z_0,z_2)}
\lambda_1^+(d_1)\lambda_2^+(d_2)\#\W^*_{\dv}. 
\end{eqnarray*}
However $\#\W^*_{\dv}=S\big(\W_{\dv};(z_0,z_0)\big)$, whence (\ref{FLa})
shows that
\[S(\W; \zv)\le XV(z_0,h^*)\Sigma+O(E_1)+O(E_2),\]
where
\[\Sigma=\sum_{d_1|P(z_0,z_1)}\;\;\sum_{d_2|P(z_0,z_2)}\lambda_1^+(d_1)
\lambda_2^+(d_2)h(\dv)\]
and the error terms are
\[E_1=Xe^{-s}\sum_{d_1<D_1}\;\;\sum_{d_2<D_2}h(\dv)\]
and
\[E_2=\sum_{f_1\leq D_1 z_0^s}\;\;\sum_{f_2\leq D_2 z_0^s}\tau^2(f_1)
\tau^2(f_2)|r(f_1,f_2)|.\]
(We write $\tau(\ldots)$ for the divisor function as usual.)

To estimate $\Sigma$ we wish to replace $h(\dv)$ by
$h_1(d_1)h_2(d_2)$. These are equal when $d_1$ and $d_2$ are
coprime. Otherwise we note that 
\beql{rhob}
h(d_1,d_2)\le C_0^{\omega(d_1d_2)}(d_1d_1)^{-1}\ll
\tau(d_1)^{C_0}\tau(d_2)^{C_0}(d_1d_2)^{-1}\
\eeq
by (\ref{rb}), and similarly $h_1(d_1)\ll \tau(d_1)^{C_0}d_1^{-1}$
and $h_2(d_2)\ll \tau(d_2)^{C_0}d_2^{-1}$. Hence if $d_1$ and $d_2$ are
not coprime then 
\[h(d_1,d_2)=h(d_1,1)h(1,d_2)+
O(\tau(d_1)^{C_0}\tau(d_2)^{C_0}(d_1d_2)^{-1}).\]
This latter case
will only hold if there is a prime $p\ge z_0$ which divides both $d_1$
and $d_2$.  As a result we may deduce that
\begin{eqnarray*}
\Sigma&=&
\sum_{d_1|P(z_0,z_1)} \sum_{d_2|P(z_0,z_2)}\lambda^+_1(d_1)
\lambda^+_2(d_2)h_1(d_1)h_2(d_2)\\
&&\hspace{1cm}\mbox{}+O\left(\sum_{p\ge z_0}\;\sum_{e_1<D_1/p}\;
\sum_{e_2<D_2/p}\tau(pe_1)^{C_0}\tau(pe_2)^{C_0}(p^2e_1e_2)^{-1}\right).
\end{eqnarray*}
The leading term factors as
\[\left\{\sum_{d_1|P(z_0,z_1)}\lambda_1^+(d_1)h_1(d_1)\right\}
\left\{\sum_{d_2|P(z_0,z_2)}\lambda_2^+(d_2)h_2(d_2)\right\}\]
and so if $h_1$ and $h_2$ satisfy the condition (\ref{d1}) the
inequalities (\ref{lsub1}) and (\ref{lsub2}) will lead to an upper bound
\[\left\{F(s_1)+O_L\left((\log z_1)^{-1/6}\right)\right\}
\left\{F(s_2)+O_L\left((\log z_2)^{-1/6}\right)\right\}V_1V_2,\]
with
\[V_i=\prod_{z_0\le p<z_i}(1-h_i(p)),\quad(i=1,2).\]
The error term is 
\[\ll\sum_{p\ge z_0}p^{-2}(\log z_1)^{2^{C_0}}(\log z_2)^{2^{C_0}}
\ll z_0^{-1}(\log z_1z_2)^{2^{1+C_0}}.\]
Hence if we take
\[z_0=\exp(\sqrt[3]{\log z_1z_2})\]
then we find that
\[\Sigma\le F(s_1)F(s_2)V_1V_2\{1+O\big((\log z_1z_2)^{-1/6}\big)\}\]
on observing that $V_i\gg (\log z_1z_2)^{-1}$, by (\ref{d1}). 

The error term $E_1$ is easily handled using (\ref{rhob}).  This produces
\[E_1\ll Xe^{-s}(\log z_1z_2)^{2^{1+C_0}}\ll X\exp\{-(\log z_1z_2)^{1/4}\}\]
on choosing
\[s=\sqrt[3]{\log z_1z_2}.\]
The bound (\ref{r*b}) shows that $V(z_0,h^*)\gg (\log z_0)^{-C_0}$,
whence we may conclude that
\[E_1\ll XV(z_0,h^*)V_1V_2F(s_1)F(s_2)(\log z_1z_2)^{-1/6}.\]
Moreover if we write $D=D_1D_2$ we have 
\[E_2\ll_{\ep}\sum_{d_1<D_1(z_1z_2)^{\ep}}\;\;\sum_{d_2<D_2(z_1z_2)^{\ep}}
\tau(d_1d_2)^4|r(d_1,d_2)|\ll_{\ep}\sum_{d_1d_2<D^{1+\ep}}
\tau(d_1d_2)^4|r(d_1,d_2)|\]
for any fixed $\ep>0$.

We can therefore summarize our result as the first statement in the
following proposition.
\begin{prop}\label{P1}
Suppose that $h(\dv)$ satisfies (\ref{rb}) and that $h_1(d)$ and
$h_2(d)$ both satisfy (\ref{d1}).  Assume further that 
$D=z_1^{s_1}z_2^{s_2}$ with $1\le s_1,s_2\ll 1$, and $\log z_1 \asymp \log z_2$ as in our setup of the Vector Sieve.  Then
\begin{eqnarray*}
S(\W; \zv)&\le& XV(z_0,h^*)V_1V_2F(s_1)F(s_2)
\{1+O\big((\log D)^{-1/6}\big)\}\\
&&\hspace{1cm}\mbox{}+O_{\ep}\left(\sum_{d_1d_2<D^{1+\ep}}
\tau(d_1d_2)^4|r(d_1,d_2)|\right)
\end{eqnarray*}
for any fixed $\ep>0$. Indeed if we write
\[\sigma_i=\frac{\log D}{\log z_i},\quad\quad (i=1,2)\]
we may replace
$F(s_1)F(s_2)$ by
\[F(\sigma_1,\sigma_2):=\inf\{F(s_1)F(s_2):\,
s_1/\sigma_1+s_2/\sigma_2=1,\,s_i\ge 1\, (i=1,2)\}.\]

Similarly we have
\begin{eqnarray*}
S(\W; \zv)&\ge& XV(z_0,h^*)V_1V_2f(\sigma_1,\sigma_2)
\{1+O\big((\log D)^{-1/6}\big)\}\\
&&\hspace{1cm}\mbox{}+
O_{\ep}\left(\sum_{d_1d_2<D^{1+\ep}}\tau(d_1d_2)^4|r(d_1,d_2)|\right)
\end{eqnarray*}
for any fixed $\ep>0$, where
\begin{eqnarray*}
f(\sigma_1,\sigma_2):&=&\sup\{f(s_1)F(s_2)+f(s_2)F(s_1)-F(s_1)F(s_2):\,
s_1/\sigma_1+s_2/\sigma_2=1,\\
&&\hspace{2cm} s_i\ge 2\, (i=1,2)\}. 
\end{eqnarray*}
\end{prop}

The lower bound is proved along the same lines as the upper bound,
using (\ref{vl}) in place of (\ref{vu}).  In handling the expression
corresponding to $\Sigma$ we encounter a leading term of the form
\[\Sigma_1^-\Sigma_2^+ + \Sigma_2^-\Sigma_1^+ -
\Sigma_1^+\Sigma_2^+,\]
where
\[\Sigma_i^{\pm}=\sum_{d|P(z_0,z_i)}\lambda_i^+(d)h_i(d).\]
In general, if 
\beql{re}
U_i\ge\Sigma_i^+\ge L_i\ge 0\quad\mbox{and}\quad\Sigma_i^-\ge L_i\quad
\mbox{for}\quad i=1,2
\eeq
then
\begin{eqnarray*}
\Sigma_1^-\Sigma_2^+ + \Sigma_2^-\Sigma_1^+ -
\Sigma_1^+\Sigma_2^+&\ge& L_1\Sigma_2^+ + L_2\Sigma_1^+ -
\Sigma_1^+\Sigma_2^+\\
&=& L_1L_2-(\Sigma_1^+ - L_1)(\Sigma_2^+ - L_2).
\end{eqnarray*}
Since $\Sigma_2^+ - L_2\ge 0$ and $U_1-L_1\ge 0$
the above expression is at least
\[L_1L_2-(U_1 - L_1)(\Sigma_2^+ - L_2)\ge 
L_1L_2-(U_1 - L_1)(U_2 - L_2)=L_1U_2+L_2U_1-U_1U_2.\]
To complete the proof of the proposition we apply the above
inequality with
\[U_i=\left\{F(s_i)+O_L\left((\log D_i)^{-1/6}\right)\right\}V_i
\quad\mbox{and}\quad
L_i=\left\{f(s_i)+O_L\left((\log D_i)^{-1/6}\right)\right\}V_i,\]
the required inequalities (\ref{re}) following from our description of
the linear sieve, given in subsection \ref{LS}, and noting that $\log D_i \asymp \log D$ for $i = 1, 2$.

\subsection{Selberg's sieve}\label{ss:ss}
Let $\W$ be a set of positive integers and for each prime $p<z$ let
$\Omega(p)$ be a set of residue classes modulo $p$.
We would like to estimate 
\[S(\W;z) = \#\{w\in \W:\,w\not\in \Omega(p) \mbox{ for all } p<z\}\]
using Selberg's sieve.

Suppose that
\beql{rem}
\#\{n\in\W: [n\mbox{ mod }p]\in\Omega(p)\mbox{ if }p\mid 
d\mbox{ and }p<z\}=h(d)X+r_d
\eeq
for some multiplicative function $h(d)\in[0,1)$.
Then the usual analysis of Selberg's sieve (see Halberstam
and Richert \cite[Theorem 3.2]{HR}, for example) shows that
\[S(\W;z)\le \frac{X}{G(z)}+\sum_{d<z^2}3^{\omega(d)}|r_d|,\]
in which
\[G(z)=\sum_{d<z}\mu^2(d)g(d)\]
where $g$ is the multiplicative function supported on squarefree numbers defined by 
\[g(p)=\frac{h(p)}{1-h(p)}.\]

For our applications we will have 
\begin{equation}
ph(p) = 
\begin{cases}
2+O(p^{-1}), & p<z_2,\\
1+O(p^{-1}), & z_2 \le p <z_1,\\
0, & \textup{otherwise,}
\end{cases}
\end{equation}
where $2\le z_2<z_1\le z$.

We now need to develop the asymptotics for $G(z)$.   
\begin{prop}\label{prop:selmult}
Preserve notation as above, and define 
\[s_i=\frac{\log z}{\log z_i}.\]
Let  
$$\rho: (0, \infty) \rightarrow \mathbb{R}$$ 
be Dickman's function, defined by 
\[\rho(s) = \left\{\begin{array}{ll}0, & s\le 0,\\
1, & 0<s\le 1, \end{array}\right.\]
and 
\begin{align}
s \rho'(s) = -\rho(s-1)
\end{align}for $s>1$.  Further let
$$B: (0,\infty)^2 \rightarrow \mathbb{R}$$ be defined by
 \begin{equation}
B(s_1, s_2)^{-1} =e^{-2\gamma}
\int\int_{\{(w_1.w_2): w_1/s_1+w_2/s_2\le 1\}}\rho(w_1)\rho(w_2)dw_1dw_2.
\end{equation}  
Then we have that
\begin{equation}
G(z)^{-1}\sim B(s_1,s_2)V(z,h)
\end{equation}
if $1\le s_1,s_2\ll 1$.
\end{prop}
We delay the proof of this result until \S \ref{sec:multfunc}. 
Note that the level of distribution required will be $D=z^2$, and that
we have taken $s_i=(\log D)/(2\log z_i)$, rather than the more normal $s_i=(\log
D)/(\log z_i)$. It is easy to translate to the latter notation, but
the definition of $B(s_1,s_2)$ would look rather less natural. 

\subsection{A version of the Bombieri--Vinogradov Theorem}
In the previous sub-sections we introduced remainder terms which can be
bounded in our applications by using a suitable version of the
Bombieri--Vinogradov Theorem.  We
begin by stating a convenient result from the literature.

\begin{lem}\label{lem:BV}
For $z_1, z_2,...,z_r \geq 2$, define the set with multiplicities
\begin{equation}\label{eqn:Pz}
P(z_1,...,z_r) = \{p^{(r)} = p_1...p_r: p_1\geq z_1,...,p_r\geq z_r\}.
\end{equation}Let $\pi_r(x; q, a)$ be the number of $p^{(r)} \in
P(z_1,...,z_r)$ such that $p^{(r)} \equiv a\pmod{q}$ and $p^{(r)} \le x$.  
Further let $\pi_r(x; q)$ be the number of $p^{(r)} \in
P(z_1,...,z_r)$ such that $p^{(r)} \le x$ and $(p^{(r)}, q) = 1$.
Then for any $A>0$ there exists $B=B(A)>0$ such that 
\begin{equation}
\sum_{q<x^{1/2}(\log x)^{-B}} \max_{(a, q) = 1} \left| \pi_r(x; q, a)
  - \frac{1}{\phi(q)}\pi_r(x; q)\right| \ll x(\log x)^{-A}, 
\end{equation}
where the implied constant depends only on $r$ and $A$.  
\end{lem}
This is Theorem 22.3 of Friedlander and Iwaniec \cite{FI}. 

Note that the result reduces to the classical version of the
Bombieri--Vinogradov Theorem when $r=1$.  In our applications we will
sometimes need to replace the set defined in \eqref{eqn:Pz} with sets
of the form 
\begin{equation}\label{eqn:PZgen}
P(z_1,...,z_r,y_1,...,y_r) = \{p^{(r)} = p_1...p_r: y_1\geq p_1\geq
z_1,...,y_r\geq p_r\geq z_r\} 
\end{equation}for $z_1,...z_r,y_1,...,y_r \geq 2$ where we allow $y_i
= \infty$ in which case the condition $y_i\ge p_i$ is automatically
fulfilled.  The lemma clearly holds for these sets as well
since we may express a set of the form \eqref{eqn:PZgen} in terms of
sets of the form \eqref{eqn:Pz}, using the inclusion-exclusion principle. 

We will actually need the following slightly different version of the
above lemma.
\begin{lem}\label{lem:BV2}
Let $P(z_1,...,z_r, y_1,...,y_r)$ be as in \eqref{eqn:PZgen} and fix
notation as in Lemma \ref{lem:BV}.  For each $q\geq 1$, let 
\begin{equation}
R_q(x) =  \max_{(a, q) = 1} \left| \pi_r(x; q, a) - \frac{1}{\phi(q)}
  \pi_r(x; q)\right|. 
\end{equation}
Then for any $A>0$ and $k\geq 1$ there exists $B = B(A,k)>0$ such that 
\begin{equation}\label{eqn:lemBV2}
\sum_{q<x^{1/2}(\log x)^{-B}} \tau(q)^k R_q(x) \ll x (\log x)^{-A},
\end{equation}
where the implied constant depends only on $r,k$ and $A$.  
\end{lem}
\begin{proof}
For $q<x$ we have  
\begin{equation}
R_q(x) \ll \frac{x}{\phi(q)},
\end{equation}
so that
\begin{equation}\label{eqn:BVRq1}
\sum_{q<x} \tau(q)^{2k} R_q(x) \ll x(\log x)^{2^{2k}}.
\end{equation}
On the other hand, Lemma \ref{lem:BV} shows that for any $A'>0$ we
will have
\begin{equation}\label{eqn:BVRq2}
\sum_{q<Q} R_q(x) \ll x(\log x)^{-A'}
\end{equation}
for $Q = x^{1/2}(\log x)^{-B'(A')}$.
The result then follows from \eqref{eqn:BVRq1} and \eqref{eqn:BVRq2} by
applying the Cauchy Schwarz inequality to \eqref{eqn:lemBV2}, and choosing $A'$
sufficiently large in terms of $A$ and $k$.   
\end{proof}

\section{Proof of the Theorem}\label{Proof}

\subsection{Bounding $S(\A; \xi_1)$}
We now take
\[\xi_i=x^{\theta_i}\;\;\;(i=1,2)\;\;\;\mbox{and}\;\;\; y=x^{\theta},\]
where $\theta_1$, $\theta_2$ and $\theta$ are constants satisfying
\[0<\theta_2<\theta_1<\frac{1}{3}\;\;\;\mbox{and}\;\;\;\theta_2<\theta<1.\]
Thus $\theta=v^{-1}$ in the notation of \S 2.

We will apply the lower bound vector sieve to the set
\[\W=\{(p+2,p+6):\, x^{1/3}<p\le x-6\},\]
taking $\zv=(\xi_1,\xi_2)$ and $X=\pi(x)$. Since $p>6$ we see that we
cannot have $d_1\mid p+2$ and $d_2\mid p+6$ unless
$(d_1,d_2)=(d_1,2)=(d_2,6)=1$. We therefore set
\[h(\dv) = 
\begin{cases}
\frac{1}{\phi(d_1d_2)}, &\textup{if } (d_1, d_2) = (d_1, 2) = (d_2, 6) =1,\\
0, &\textup{otherwise},
\end{cases}\]
and 
\[\#\W_{\dv} = h(\dv) \pi(x) + R(\dv),\]
whence Lemma \ref{lem:BV2} gives
\begin{equation}\label{nl}
\sum_{\substack{\dv\\d_1d_2<x^{1/2-\ep}}} \tau(d_1d_2)^4|R(\dv)| \ll x(\log x)^{-A},
\end{equation}
for any positive constant $A$.

We now use the vector sieve lower bound from Proposition \ref{P1}.  
According to (\ref{nl}), the remainder sum can be bounded adequately
when $D=x^{1/2-2 \epsilon}$.  The Euler factors in $V(z_0,h^*)V_1V_2$
are
\[1=4\big(1-\tfrac{1}{2}\big)^2\quad\mbox{for}\quad p=2,\]
\[1-\tfrac{1}{2}=\tfrac{9}{8}\big(1-\tfrac{1}{3}\big)^2\quad
\mbox{for}\quad p=3,\]
\[1-\frac{2}{p-1}=\left(1-\frac{3p-1}{(p-1)^3}\right)
\left(1-\frac{1}{p}\right)^2\quad\mbox{for}\quad 5\le p<z_0,\]
\[\left(1-\frac{1}{p-1}\right)^2=\left(1-\frac{1}{(p-1)^2}\right)^2
\left(1-\frac{1}{p}\right)^2\quad\mbox{for}\quad z_0\le p<\xi_2,\]
and
\[\left(1-\frac{1}{p-1}\right)=\left(1-\frac{1}{(p-1)^2}\right)
\left(1-\frac{1}{p}\right)\quad\mbox{for}\quad \xi_2\le p<\xi_1,\]
whence
\[V(z_0,h^*)V_1V_2\sim CV(\xi_1)V(\xi_2)\]
with
\beql{eqn:C}
C=\frac{9}{2}\prod_{p>3}\left(1-\frac{3p-1}{(p-1)^3}\right)
\eeq
and
\[V(z)=\prod_{p<z}(1-p^{-1})\sim \frac{e^{-\gamma}}{\log z}.\]
Note that $C$ is the constant appearing in the Hardy-Littlewood conjectures for such prime tuples.  We therefore obtain the lower bound
\beql{eqn:mainterm}
S(\A; \xi_1) \geq (C+o(1)) \pi(x) V(\xi_1)V(\xi_2)
f\big((2\theta_1)^{-1},(2\theta_2)^{-1}\big).
\eeq

\subsection{The terms $S(\A_p; \xi_1)$}
We may apply the upper bound vector sieve with the same set $\W$ as
before, noting that $h(pd_1,d_2)=\phi(p)^{-1}h(\dv)$ when $d_1\mid
P(\xi_1)$, $d_2|P(\xi_2)$ and $p\ge \xi_1$.  This easily leads to the bound
\begin{eqnarray*}
S(\A_p; \xi_1)&\le& (C+o(1)) 
\frac{\pi(x)}{p-1} V(\xi_1)V(\xi_2)F(s_1(p),s_2(p))\\
&&\hspace{1cm}\mbox{}+
O_{\ep}\left(\sum_{d_1d_2<p^{-1}D^{1+\ep}}\tau(d_1d_2)^4|R(pd_1,d_2)|\right),
\end{eqnarray*}
with
\[s_i(p)=\frac{\log D/p}{\log \xi_i}\quad (i=1,2).\]
Hence for any $P \le P'$ we have
\begin{eqnarray*}
\sum_{P\le p\le P'}S(\A_p; \xi_1)&\le& (C+o(1))\pi(x)\left(\sum_{P\le
    p\le P'}\frac{F(s_1(p),s_2(p))}{p-1}\right)V(\xi_1)V(\xi_2)\\
&&\hspace{1cm}\mbox{}
+O_{\ep}\left(\sum_{pd_1d_2<D^{1+\ep}}\tau(d_1d_2)^4|R(pd_1,d_2)|\right).
\end{eqnarray*}
The remainder sum is negligible, by Lemma \ref{lem:BV2}, if $D =
x^{1/2-2\ep}$.  It follows that
\begin{equation}\label{ext}
\sum_{P\le p\le P'}S(\A_p; \xi_1)\le (C+o(1))\pi(x)V(\xi_1)V(\xi_2)
\int_P^{P'}\frac{F(\sigma_1(t),\sigma_2(t))}{t\log t}dt,
\end{equation}
where we now have
\beql{sigdef}
\sigma_i(t)=\frac{\log \sqrt{x}/t}{\log \xi_i}\quad (i=1,2).
\eeq

Alternatively we can use Selberg's sieve as in subsection \ref{ss:ss}.  
For a given prime $p$ we take $\W=\W^{(p)}$ to consist of the 
values $(q+2)(q+6)$ where $q$ runs over primes in the interval
$x^{1/3}<q\le x-6$ such that $p\mid q+2$. For each
prime $r$ we use the residue classes
\begin{equation}\label{l1}
\Omega(r) = 
\begin{cases}
\emptyset, & r=2,\\
\{-2\}, & r=3,\\
\{-2, -6\} &  5\leq r <\xi_2,\\
\{-2\}, &  \xi_2\le r<\xi_1\\
\emptyset, & r\ge \xi_1.
\end{cases}
\end{equation}
It is natural to take $X=\pi(x)/(p-1)$ and
\beql{l2}
h(r) = 
\begin{cases}
0, & r=2,\\
\frac{1}{2}, & r=3\\
\frac{2}{r-1}, & 5\le r< \xi_2,\\
\frac{1}{r-1}, & \xi_2\le r<\xi_1,\\
0, & r\ge \xi_1.
\end{cases}
\eeq
Let $z\ge \xi_1$ and write 
\[s_i=\frac{\log z}{\log \xi_i}\quad i=1,2\]
Then
\[V(z,h)\sim CV(\xi_1)V(\xi_2).\]
If we write $r_d^{(p)}$ for the corresponding remainder in
(\ref{rem}) we will have
\begin{eqnarray*}
S(\A_p; \xi_1)&=&S(\W;z)\\
&\le& \frac{\pi(x)}{(p-1)G(z)}+\sum_{d<z^2}3^{\omega(d)}|r_d^{(p)}|\\
&=& (C+o(1))\frac{\pi(x)}{p-1}B(s_1,s_2)V(\xi_1)V(\xi_2)+
\sum_{d<z^2}3^{\omega(d)}|r_d^{(p)}|.
\end{eqnarray*}
Moreover
\[|r_d^{(p)}|\le\tau(d)\max_{(a,pd)=1}\left|
\{\pi(x-6;pd,a)-\pi(x^{1/3};pd,a)\}-\frac{\pi(x-6)-\pi(x^{1/3})}{\phi(q)}
 \right|.\]
Lemma \ref{lem:BV2} then shows that if we choose
$z=(\sqrt{x}/P)^{1/2}(\log x)^{-C_0}$ with a suitably large constant $C_0$ then
\[\sum_{P\le p\le P'}S(\A_p; \xi_1)\le (C+o(1))\pi(x)\left(\sum_{P\le
    p\le P'}\frac{B(\tilde \sigma_1(p),\tilde \sigma_2(p))}{p-1}\right)V(\xi_1)V(\xi_2),\]
with 
$$\tilde \sigma_i(t) = \frac{\log \bfrac{\sqrt{x}}{t}^{1/2}}{\log \xi_i},$$ provided that $\xi_1\le z$.
We then deduce that
\[\sum_{P\le p\le P'}S(\A_p; \xi_1)\le (C+o(1))\pi(x)V(\xi_1)V(\xi_2)
\int_P^{P'}\frac{B(\tilde \sigma_1(t),\tilde \sigma_2(t))}{t\log t}dt.\]
Comparison with (\ref{ext}) now shows that
\begin{equation}\label{eqn:middleterms}
\sum_{\xi_1\le p\le x^{1/3}}S(\A_p,\xi_1)\le
(C+o(1))\pi(x)V(\xi_1)V(\xi_2)I(\theta_1, \theta_2),
\end{equation}
with
\beql {eqn:Idef}
I(\theta_1, \theta_2)=
\int_{\theta_1}^{1/3}\alpha^{-1}
\min\left\{F\big(\frac{1-2\alpha}{2\theta_1}\,,\,
\frac{1-2\alpha}{2\theta_2}\big)\,,\,B\big(\frac{1-2\alpha}{4\theta_1}\,,\,
\frac{1-2\alpha}{4\theta_2}\big)\right\}d\alpha.
\eeq

\subsection{Estimating $N_0$ via Chen's r\^ole-reversal trick}

The number $N_0$ is defined in terms of products $p_1p_2p_2\in\A$.
However we can change our point of view and write
\[N_0=\#\{p+2\in\B: p\mbox{ prime}\},\]
where 
\begin{eqnarray*}
\B&=& \{p_1p_2p_3:\,\xi_1\le p_1\le x^{1/3}<p_2<p_3,\,
x^{1/3}+2<p_1p_2p_3\le x-4,\\
&& \hspace{3cm}(p_1p_2p_3+4,P(\xi_2))=1\}.
\end{eqnarray*}
Thus instead of sieving numbers $p+2$ and $p+6$ we will sieve numbers
$p_1p_2p_3-2$ and $p_1p_2p_3+4$. This is Chen's reversal of r\^oles.
Following the approach of subsection \ref{ss:ss}, we let
\[\W=\{p_1p_2p_3:\,\xi_1\le p_1\le x^{1/3}<p_2<p_3,\mbox{ and }
p_1p_2p_3\le x-4\}\]
and for each prime $r$ we define the set $\Omega(r)$
by
\begin{equation}\label{l3}
\Omega(r) = 
\begin{cases}
\emptyset, & r=2,\\
\{2\}, & r=3,\\
\{2, -4\} &  5\leq r <\xi_2,\\
\{2\}, &  \xi_2\le r<z\\
\emptyset, & r\ge z.
\end{cases}
\end{equation}
It follows that
\[N_0\le S(\W;z)\]
for any $z$ between $\xi_2$ and $x^{1/4}$, say.

It is natural to take $X=\#\W$ and to choose the function  
$h(r)$ to be given by (\ref{l2}) as 
before, except that now $h(r)=1/(r-1)$ for $\xi_2\le r<z$. 
With this definition we will have
\[V(z,h)\sim CV(z)V(\xi_2)\sim CV(\xi_1)V(\xi_2)s_1^{-1},\]
where $s_1=(\log z)/(\log \xi_1)$.
Moreover if we define $r_d$ via (\ref{rem}) then we can use Lemma
\ref{lem:BV2} with $z=x^{1/4}(\log x)^{-B/2}$ to show that
\beql{rest1}
\sum_{d<z^2}3^{\omega(d)}|r_d|\ll x(\log x)^{-A}.
\eeq
In order to do this we replace $\W$ by the set
\[\W_0=\{p_1p_2p_3:\,\xi_1\le p_1\le x^{1/3}<p_2,p_3,
\mbox{ and }p_1p_2p_3\le x-4\},\]
to which Lemma \ref{lem:BV2} applies directly. We should also note that
\[\pi_r(x-4;q)=\#\W_0+O(x\xi_1^{-1})=\#\W_0+O(x(\log x)^{-A-1}),\]
on allowing for possible common factors of $q$ and $p_1p_2p_3$. The
error term here is certainly small enough for (\ref{rest1}).

We also need to estimate $\#\W$.  We find that 
\begin{align} \label{eqn:Best}
\#\W  \sim \sum_{\substack{\xi_1 \leq p_1 \leq x^{1/3} \\ x^{1/3} \leq 
    p_2 \leq (x/p_1)^{1/2}}}  \frac{x}{p_1p_2 \log \frac{x}{p_1p_2}}  
&\sim x \int_{\xi_1}^{x^{1/3}} \int_{x^{1/3}}^{(x/v)^{1/2}}
\frac{du dv}{(u\log u)(v\log v)\log \bfrac{x}{uv}} \notag \\  
&\sim \pi(x) L(\theta_1^{-1}) 
\end{align}where
\begin{equation} \label{eqn:Gdef}
L(s) = \int_{1/s}^{1/3} \int_{1/3}^{\frac{1-\beta}{2}} \frac{d\alpha
  d\beta}{\alpha \beta (1-\alpha - \beta)}, 
\end{equation} by the change of variables $u = x^\alpha$ and $v = x^\beta$.  
We therefore conclude that
\beql{eqn:SBz2}
N_0\le \big(C+o(1)\big)\pi(x)V(\xi_1)V(\xi_2)L\big(\theta_1^{-1}\big)
4\theta_1 B\big(1,(4\theta_2)^{-1}\big).
\eeq

It is possible as well to apply the vector sieve here, but the bound
\eqref{eqn:SBz2} is always superior for our application. 

\subsection{The weighted sieve terms}
We now turn our attention to 
\begin{equation} \label{eqn:T}
T := \sum_{\xi_2\le q\le y}w_q\left(\#\B_q^{(1)}+\frac12\#\B_q^{(2)}\right).
\end{equation}

We write
\beql {eqn:Tupperbd}
T \le \sum_{\xi_2\le q\le y}w_q\left(\#\V_q^{(1)}+\frac12\#\V_q^{(2)}\right),
\eeq
where
\begin{equation}
\V^{(1)} = \{n+4: n \in \A \textup{ and } (n, P(x^{1/4})) = 1 \},
\end{equation}and
\begin{equation}
\V^{(2)} = \{n+4: n \in \A \textup{ and } n = p_1p_2, \xi_1 \le p_1 \le x^{1/4} < p_2 \}.
\end{equation}Note that $\B^{(1)} \subset \V^{(1)}$ and every element of $B^{(2)}$ is in $\V^{(2)}$ with the exception of those $n+4$ with $n \in \A$, $n = p_1p_2$ with $p_1, p_2 > x^{1/4}$, and those are counted in $\V^{(1)}$.

$\\$
We begin by examining $\#\V_q^{(1)}$.  We will use the Selberg sieve
as in subsection \ref{ss:ss}.  To be precise, we take
\[\W=\{p\in(x^{1/3},x-6]:
\,q\mid p+6\}\]
and
\[\Omega(r) = \begin{cases}
\emptyset, & r=2\mbox{ or }q,\\
\{-2\}, & r=3,\\
\{-2, -6\} &  5\leq r <\xi_2,\\
\{-2\}, &  \xi_2\le r<z,\,r\neq q.
\end{cases}\]
For some $z \in [\xi_2, x^{1/4}]$, we choose
\[X=\frac{1}{q-1}\#\{p: x^{1/3} < p \le x-6  \}\]
and use $h$ given by 
\begin{equation} \label{eqn:hrdef}
h(r) = 
\begin{cases}
0, & r=2\mbox{ or }q,\\
\frac{1}{2}, & r=3\\
\frac{2}{r-1}, & 5\le r< \xi_2,\\
\frac{1}{r-1}, & \xi_2\le r<z,\,r\neq q,
\end{cases}
\end{equation}
whence
\[V(z,h)\sim CV(z)V(\xi_2).\]
In estimating the individual terms in
\[\sum_{P<q\le P'}\#\V_q^{(1)},\]
Lemma \ref{lem:BV2} will allow us to use
$z=x^{1/4}q^{-1/2}(\log x)^{-B/2}$, provided that $z\ge\xi_2$. 
We therefore conclude that
\[\sum_{P<q\le P'}\#\V_q^{(1)}\leq \big(C+o(1)\big)\pi(x)
\sum_{P<q\le  P'}\frac{1}{q-1}V(z)V(\xi_2)
B\left(1,\frac{\log z}{\log \xi_2}\right).\]
Since $\log z\sim\log(x^{1/4}q^{-1/2})$ we have
\[V(z)\sim V(\xi_1)\frac{\log \xi_1}{\log(x^{1/4}q^{-1/2})},\]
and we deduce that
\[\sum_{\xi_2\le q\le y}w_q\#\V_q^{(1)}\le \big(C+o(1)\big)\pi(x) 
V(\xi_1)V(\xi_2)\Sigma,\]
with
\[\Sigma=\sum_{\xi_2\le q\le y}\frac{w_q}{q-1}
\frac{\log \xi_1}{\log(x^{1/4}q^{-1/2})}
B\left(1,\frac{\log x^{1/4}q^{-1/2}}{\log \xi_2}\right).\]
In order to ensure that $z \ge \xi_2$, we impose the condition 
that $2\theta_2+\theta < 1/2$.  Bearing in mind the  
definition (\ref{wtd}) of the weights $w_q = 1 - \frac{\log q}{\log y}$ we apply the Prime Number Theorem to see that
\begin{align*}
\Sigma &\sim \int_{\xi_2}^{y} \frac{\log \xi_1}{\log (x^{1/4}q^{-1/2})}\left(1-\frac{\log t}{\log y}\right) B\left(1, \frac{\log (x^{1/4}t^{-1/2})}{\log \xi_2}\right) \frac{dt}{t\log t} \\
&\sim \int_{\theta_2}^{\theta}\frac{4\theta_1}{1-2\alpha}
\frac{\theta-\alpha}{\alpha\theta}
B\left(1,\frac{1-2\alpha}{4\theta_2}\right)d\alpha.
\end{align*}For notational convenience, let
\begin{equation}\label{eqn:Jdef}
J(\theta_1, \theta_2, \theta) := \int_{\theta_2}^{\theta}\frac{4\theta_1}{1-2\alpha}
\frac{\theta-\alpha}{\alpha\theta}
B\left(1,\frac{1-2\alpha}{4\theta_2}\right)d\alpha,
\end{equation}so that
\begin{equation}\label{eqn:Vq1}
\sum_{\xi_2\le q\le y}w_q\#\V_q^{(1)}\le \big(C+o(1)\big) J(\theta_1, \theta_2, \theta)\pi(x)
V(\xi_1)V(\xi_2),
\end{equation}if $2\theta_2+\theta <1/2$.

As in the previous section, here too we could have used the vector sieve upper
bound, but again the Selberg method is superior. 
$\\$

We now examine $\#\V_q^{(2)}$.  Again, we will use the Selberg sieve
as in subsection \ref{ss:ss}, but our approach 
to $\V_q^{(1)}$ and our approach to $\V_q^{(2)}$ differ.  
In our treatment of $\V_q^{(1)}$, we 
took $\W$ to be a set of primes $p$ and used the sieve to handle the
conditions that $(p+2, P(x^{1/4})) = 1$ and $(p+6, P(\xi_2)) = 1$.
Here, we will take $\W$ to be a set of numbers of the form $n=p_1p_2$
for $p_1$ and $p_2$ prime, and use the sieve to handle the condition
that $n-2$ is prime and $(n+4, P(\xi_2)) = 1$ 

To be precise, we take
\[\W=\{p_1p_2\in(x^{1/3}+2,x-4]:
\,q\mid p_1p_2+4,\, \xi_1\le p_1\le x^{1/4}<p_2\}\]
and
\[\Omega(r) = \begin{cases}
\emptyset, & r=2\mbox{ or }q,\\
\{2\}, & r=3,\\
\{2, -4\} &  5\leq r <\xi_2,\\
\{2\}, &  \xi_2\le r<z,\,r\neq q.
\end{cases}\]
We choose
\[X=\frac{1}{q-1}\#\{p_1p_2\in(x^{1/3}+2,x-4]:
\, \xi_1\le p_1\le x^{1/4}<p_2\}\]
and use $h$ given by \eqref{eqn:hrdef} as before.  Recall that
\[V(z,h)\sim CV(z)V(\xi_2).\]
We have
\begin{align*}
X 
&\sim \frac{1}{q-1} \sum_{\xi_1 \le p_1 \le x^{1/4}} \frac{x}{p_1 \log \bfrac {x}{p_1}} \\
&\sim \frac{1}{q-1} \int_{\xi_1}^{x^{1/4}} \frac{x}{t\log \bfrac x t} \frac{dt}{\log t}\\
&\sim \frac{1}{q-1} \pi(x) \int_{\theta_1}^{1/4} \frac{du}{u(1-u)}\\
&= \frac{1}{q-1} \pi(x) \left(\log \frac {1-\theta_1}{3\theta_1}\right).
\end{align*}

In estimating
\[\sum_{P<q\le P'}\#\V_q^{(2)},\]
Lemma \ref{lem:BV2} will again allow us to use
$z=x^{1/4}q^{-1/2}(\log x)^{-B/2}$, provided that $z\ge\xi_2$. 
We therefore conclude that
\[\sum_{P<q\le P'}\#\V_q^{(2)}\leq \big(C+o(1)\big)\pi(x) \log \bfrac{1-\theta_1}{3 \theta_1}
\sum_{P<q\le  P'}\frac{1}{q-1}V(z)V(\xi_2)
B\left(1,\frac{\log z}{\log \xi_2}\right).\]
Continuing as in the previous section, we have
\beql {eqn:Vq2}
\sum_{\xi_2\le q\le y}w_q\#\V_q^{(2)}\le \big(C+o(1)\big) \log \bfrac{1-\theta_1}{3 \theta_1} J(\theta_1, \theta_2, \theta) \pi(x)
V(\xi_1)V(\xi_2),
\eeq
where $J$ is as defined in \eqref{eqn:Jdef}.  As in the previous section, here too we could have used the vector sieve upper bound, but again the Selberg method is superior. 

\subsection{Summary}
Putting \eqref{eqn:mainterm}, \eqref{eqn:middleterms} and \eqref{eqn:SBz2} into 
\eqref{eqn:basicdecomp} and by \eqref{eqn:Vq1}, \eqref{eqn:Tupperbd}  
and \eqref{eqn:Vq2}, we have  
\begin{equation}
S_1 - \lambda \sum_{\xi_2 \le p \le y} w_p \left(\# \B_p^{(1)} + \frac 12 \# \B_p^{(2)}\right) \geq C\pi(x) V(\xi_1)V(\xi_2)  H(\theta_1, \theta_2, \theta, \lambda) (1+o(1)), 
\end{equation}where

\begin{align}
H(\theta_1, \theta_2, \theta, \lambda) &= f((2\theta_1)^{-1}, (2\theta_2)^{-1})  - \frac {1}{2} I\left(\theta_1, \theta_2\right) - 2 L(\theta_1^{-1}) \theta_1 B(1, (4\theta_2)^{-1}) \notag \\
& - \lambda \left(1 + \frac 12
  \log \frac {1-\theta_1}{3\theta_1} \right)J(\theta_1, \theta_2, \theta). 
\end{align}

We chose $\theta_1 = 1/11$, $\theta_2 = 1/410$, $\theta = 1/30$ and $\lambda = 0.0145$.  Recall that $f$ is defined as a supremum - using Matlab to conduct a rough search for the sup, we compute that
$$f((2\theta_1)^{-1}, (2\theta_2)^{-1}) \geq 0.9992523...
$$
When calculating $I(\theta_1, \theta_2)$, the
quantity $$\min\left\{F\big(\frac{1-2\alpha}{2\theta_1}\,,\, 
\frac{1-2\alpha}{2\theta_2}\big)\,,\,B\big(\frac{1-2\alpha}{4\theta_1}\,,\,
\frac{1-2\alpha}{4\theta_2}\big)\right\}$$ appears, arising from use of both the vector sieve and our version of Selberg's sieve.  For our values of $\theta_1$ and $\theta_2$, $F\big(\frac{1-2\alpha}{2\theta_1}\,,\,
\frac{1-2\alpha}{2\theta_2}\big)$ is smaller for small values of $\alpha$, while $B\big(\frac{1-2\alpha}{4\theta_1}\,,\,
\frac{1-2\alpha}{4\theta_2}\big)$ becomes a better choice at around
$\alpha = 0.26.$  Again, a Matlab computation with a rough optimization of the value of $F\big(\frac{1-2\alpha}{2\theta_1}\,,\,
\frac{1-2\alpha}{2\theta_2}\big)$, which was defined as an inf, gives that
$$I\left(\theta_1, \theta_2\right) \leq 1.5630111...
$$It is simple to calculate that
$$L(\theta_1^{-1}) = 0.5477550...
$$

Further, using that 
\begin{equation}
\int_0^\infty \rho(u) du = \int_0^\infty u \rho(u) du = e^\gamma,
\end{equation}we have that
\begin{align*}
B(1, v)^{-1} &= e^{-2\gamma}\int_0^\infty \rho(y) \int_0^{1-y/v} 1 dx dy - \int_v^\infty \rho(y) \int_0^{1-y/v} 1 dx dy.
\end{align*}On the other hand, using the crude upper bound $\rho(n) \le \frac{1}{n!}$, we have that
\begin{equation*}
\int_v^\infty \rho(y) \int_0^{1-y/v} 1 dx dy \le \sum_{n \geq \lfloor v\rfloor v} \frac{1}{n!} \leq \frac{e}{\lfloor v\rfloor!},
\end{equation*}while
\begin{align*}
e^{-2\gamma}\int_0^\infty \rho(y) \int_0^{1-y/v} 1 dx dy 
&= \left(1 - \frac 1 v\right) e^{-\gamma}.
\end{align*}

From this, it follows that
\begin{equation*}
B(1, (4\theta_2)^{-1}) = 1.7986199...
\end{equation*}

Substituting the above estimate for $B(1, v)$ into the definition of $J(\theta_1, \theta_2, \theta)$, we find that
$$
J(\theta_1, \theta_2, \theta)  = 1.1235270...
$$

We find that $H(\theta_1, \theta_2, \theta, \lambda) >0$ for $\lambda < 0.0214$.  From \eqref{eqn:weightedsecondcompbdd} this gives a bound
for $r$ of the form $r \le 1/\theta + 1/\lambda < 77$, giving the result that there are infinitely many primes $p$ such that $p+2$ has at most $2$ prime factors and $p+6$ has at most $76$ prime factors.


\section{The average of multiplicative functions 
appearing in Selberg's sieve} \label{sec:multfunc}
We end by proving Proposition \ref{prop:selmult}.  For $i \in \{1, 2\}$, let 
\[\chi_i(n) = 
\begin{cases}
1, &\textup{ if } p|n \Rightarrow p< z_i,\\
0, &\textup{ otherwise}.
\end{cases}\]
Recall that $g$ is the multiplicative function supported on squarefree numbers defined by
\[g(p) = \frac{h(p)}{1-h(p)}.
\]
We further define the multiplicative functions $k$ and $j$ by
\[n\mu^2(n)g(n) = (\chi_1*k)(n) = (\chi_1*\chi_2*j)(n),\]
for all natural numbers $n$, so that
\begin{eqnarray*}
\sum_{n\geq 1} \frac{n\mu^2(n)g(n)}{n^s} &=& \left(\sum_{n\geq 1}
  \frac{\chi_1(n)}{n^s}\right) \left(\sum_{n\geq 1}
  \frac{k(n)}{n^s}\right)\\
&  =&\left(\sum_{n\geq 1}
  \frac{\chi_1(n)}{n^s}\right)\left(\sum_{n\geq 1}
  \frac{\chi_2(n)}{n^s}\right) \left(\sum_{n\geq 1}
  \frac{j(n)}{n^s}\right).  
\end{eqnarray*}
The Dirichlet series above clearly converge for $\tRe s
> 1$.  Moreover we see that 
\[\sum_{r\geq 1}\frac{j(r)}{r^s} = \prod_{p} 
\left(1+ \frac{g(p)}{p^{s-1}}\right) \prod_{p<z_1} \left(1-\frac{1}{p^s}\right)
\prod_{p<z_2} \left(1-\frac{1}{p^s}\right). \]
Thus 
\[j(p)\ll p^{-1},\quad\quad\mbox{and}\quad\quad j(p^e)\ll 1\;\;(e\ge 2).\]
Similarly we find that
\[k(p)=\left\{\begin{array}{rr} 1+O(p^{-1}), & p<z_2,\\ O(p^{-1}), &
    z_2\le p<z_1,\\ 0, & p\ge z_1,\end{array}\right.\quad\quad
\mbox{and}\quad\quad k(p^e)\ll 1\;\;(e\ge 2).\]
These estimates suffice to show that
\beql{js}
\sum_{R< r\le 2R}|j(r)|\ll_{\ep} R^{1/2+\ep}
\eeq
for any fixed $\ep>0$, and
\beql{eqn:kjdirichlet}
\sum_{M<m\le 2M}|k(m)|\ll M.
\eeq

In order to study $G(z)$ we will first examine the average of
$d\mu^2(d)g(d)$,  whose behaviour resembles that of
$\chi_1*\chi_2$.  We intend to take advantage of the fact that
$\chi_1$ and $\chi_2$ are indicator functions of smooth numbers, and
the computation of their averages are standard results.  We have 
\begin{align}\label{eqn:nhsum}
\sum_{n\leq x} n\mu^2(n)g(n) = \sum_{m\leq x} k(m) \Psi(\frac xm; z_1),
\end{align}where $\Psi(x; y)$ is the number of $y$-smooth numbers
below $x$.  It follows from a result of de Bruijn \cite{dB} that 
\[\Psi(x; y) = x \rho\bfrac{\log x}{\log y}+O\left(\frac{x}{\log(2x)}\right)  
\]
uniformly for $1\le y\le x$, where $\rho$ is Dickman's function defined
as in the statement of
Proposition \ref{prop:selmult}.  Continuing from \eqref{eqn:nhsum}, we
have 
\begin{align}\label{eqn:nh2}
\sum_{n\leq x} n\mu^2(n)g(n) &= x \sum_{m\leq x} \frac{k(m)}{m}
\rho\bfrac{\log x/m}{\log z_1} + O\left(x \sum_{m\le x}
  \frac{|k(m)|}{m\log(2x/m)}\right) \notag \\ 
&= x \sum_{m\leq x} \frac{k(m)}{m} \rho\bfrac{\log x/m}{\log z_1} +
O(x\log\log x),
\end{align} upon observing that 
\[\sum_{m\le x}\frac{|k(m)|}{m\log(2x/m)} \ll \log\log x\]
by \eqref{eqn:kjdirichlet}.  
A similar calculation yields 
\begin{align}\label{ka}
\sum_{m\leq y} k(m) &= \sum_{r\leq y} j(r) \Psi\left(\frac yr; z_2\right) \notag \\
& = y \sum_{r\leq y} \frac{j(r)}{r} \rho \bfrac{\log y/r}{\log
z_2} + O\left(y\sum_{r\le y}\frac{|j(r)|}{r\log(2y/r)}\right) \notag \\ 
&= y\sum_{r\leq y} \frac{j(r)}{r} \rho\bfrac{\log y/r}{\log z_2} +
O(y(\log 2y)^{-1}),
\end{align} 
after noting that 
\[\sum_{r\le y}\frac{|j(r)|}{r\log(2y/r)}\ll (\log 2y)^{-1}\]
by \eqref{js}.  Another application of \eqref{js} shows that
\[\sum_{r\leq y} \frac{j(r)}{r} \rho \bfrac{\log y/r}{\log z_2}
=\sum_{r\leq \sqrt{y}} \frac{j(r)}{r} \rho \bfrac{\log y/r}{\log z_2}
+O(y^{-1/8}).\] 
Since $\rho'(t)\ll t^{-1}$ for $t>0$ we have
$\rho\bfrac{\log y/r}{\log z_2} 
 = \rho \bfrac{\log y}{\log z_2} + O\left(\frac{\log r}{\log
2y}\right)$ for $r\le\sqrt{y}$, whence two more applications of 
\eqref{js} yield
\begin{eqnarray*}
\sum_{r\leq \sqrt{y}} \frac{j(r)}{r} \rho \bfrac{\log y/r}{\log z_2}&=& 
\rho \bfrac{\log y}{\log z_2}\sum_{r\leq \sqrt{y}} \frac{j(r)}{r}+
O\left((\log 2y)^{-1}\sum_{r\leq \sqrt{y}} \frac{|j(r)|\log r}{r}\right)\\ 
&=&\rho \bfrac{\log y}{\log z_2}\sum_{r=1}^{\infty} \frac{j(r)}{r}+
O((\log 2y)^{-1}).
\end{eqnarray*}
It therefore follows from \eqref{ka} that
\[\sum_{m\leq y} k(m)=C_0 y\rho \bfrac{\log y}{\log z_2}
+O(y(\log 2y)^{-1}),\]
where
\beql{C0E}
C_0=\sum_{r=1}^{\infty} \frac{j(r)}{r}=V(z_1)V(z_2)V(z,h)^{-1}
\sim e^{-2\gamma}(\log z_1)^{-1}(\log z_2)^{-1}V(z,h)^{-1}.
\eeq
Note here that $C_0\ll 1$, since $C_0$ can be written as a product of
Euler factors each of which is $1+O(p^{-2})$.

We may now insert the above formula into \eqref{eqn:nh2}, using partial
summation to deduce that
\begin{eqnarray*}
\sum_{n\leq x} n\mu^2(n)g(n)&=&
x \sum_{m\leq x} \frac{k(m)}{m} \rho\bfrac{\log x/m}{\log z_1} +
O(x\log\log x)\\
&=&C_0 x\int_1^x\frac{1}{t}\frac{d}{dt}\left\{t\rho\left(\frac{\log
t}{\log z_2}\right)\right\}\rho\bfrac{\log x/t}{\log z_1} dt+O(x\log\log x).
\end{eqnarray*}
The integral is
\[\int_1^x\frac{1}{t}\rho\bfrac{\log t}{\log z_2}
\rho\bfrac{\log x/t}{\log z_1} dt
+\int_1^x\frac{1}{t\log z_2}\rho'\bfrac{\log t}{\log z_2}
\rho\bfrac{\log x/t}{\log z_1}dt.\]
However $\rho'(s)=0$ for $0<s\le 1$ and $\rho'(s)\ll s^{-1}$
otherwise.  Thus the second integral above is $O(\log\log x)$ so that
\[\sum_{n\leq x} n\mu^2(n)g(n)=C_0 x
\int_1^x
\rho\bfrac{\log t}{\log z_2}\rho\bfrac{\log x/t}{\log z_1}\frac{dt}{t}
+O(x\log\log x).\]
A further summation by parts now shows that
\begin{eqnarray*}
G(z)&=&\sum_{n<z}\mu^2(n)g(n)\\
&=&C_0\left\{\int_1^z\rho\bfrac{\log t}{\log z_2}
\rho\bfrac{\log z/t}{\log z_1}\frac{dt}{t}+\int_1^z\frac{1}{x}
\int_1^x\rho\bfrac{\log t}{\log z_2}\rho\bfrac{\log x/t}{\log z_1}\frac{dt}{t}
dx\right\}\\
&&\hspace{1cm}\mbox{}+O((\log z)(\log\log z)).
\end{eqnarray*}
The first integral above is $O(\log z)$, which may be absorbed into
the error term, while the second is
\[(\log z_1)(\log z_2)\int\int_{\{(w_1,w_2):w_1,w_2\ge 0,\,w_1/s_1+w_2/s_2\le 1\}}
\rho(w_1)\rho(w_2)dw_1dw_2.\]
The proposition now follows from (\ref{C0E}).


\end{document}